\documentclass[a4paper,12pt]{scrartcl}
\usepackage[utf8x]{inputenc}
\usepackage{amsmath,amssymb,amsthm}
\usepackage{color}
\usepackage
	[
	colorlinks,
	backref
	]
{hyperref}
\usepackage[all]{xy}
\usepackage{float}

\title{
Sieves and the Minimal Ramification Problem
}
\author{Lior Bary-Soroker and Tomer M. Schlank}

\newtheorem{thmvoid}{}[section]
\newtheorem{conjecture}[thmvoid]{Conjecture}
\newtheorem{lemma}[thmvoid]{Lemma}
\newtheorem{theorem}[thmvoid]{Theorem}
\newtheorem{proposition}[thmvoid]{Proposition}
\newtheorem{corollary}[thmvoid]{Corollary}
\theoremstyle{definition}
\newtheorem{definition}[thmvoid]{Definition}

\newcommand{\disc}{\mathop{\rm disc}}

\newcommand{\Ram}{{\rm Ram}}
\newcommand{\RamG}{{\rm Branch}}

\newcommand{\Prms}{{\rm Prms}}

\newcommand{\bfd}{{\mathbf{d}}}

\newcommand{\GL}{{\rm GL}}

\newcommand{\Spec}{{\rm Spec}}

\newcommand{\ZZ}{\mathbb{Z}}
\newcommand{\RR}{\mathbb{R}}
\renewcommand{\AA}{\mathbb{A}}
\newcommand{\FF}{\mathbb{F}}
\newcommand{\QQ}{\mathbb{Q}}
\newcommand{\PP}{\mathbb{P}}

\newcommand{\Gal}{\mathop{\rm Gal}}
\begin{document}

\maketitle
\abstract{
The minimal ramification problem  may be considered  as a quantitative version of the inverse Galois problem. 
For a nontrivial finite group $G$, let $m(G)$ 
be the minimal integer $m$ for which there exists a Galois extension $N/\QQ$ that is ramified at exactly $m$ primes 
(including the infinite one). 
So, the problem is to compute or to bound $m(G)$. 

In this paper, we 
bound the ramification of extensions $N/\QQ$ 
obtained as a specialization of a branched covering $\phi\colon C\to \PP^1_{\QQ}$. This leads to novel upper bounds on $m(G)$, for finite groups $G$ that are realizable as the Galois group of a branched covering. 
Some instances of our general results are: 
\[
1\leq m(S_m)\leq 4 \quad \mbox{and} \quad n\leq m(S_m^n) \leq n+4,
\] for all $n,m>0$. Here $S_m$ denotes the symmetric group on $m$ letters, and $S_m^n$ is the direct product of $n$ copies of $S_m$. 
We also get the correct asymptotic of $m(G^n)$, as $n \to \infty$ for a certain class of groups $G$.

Our methods are based on sieve theory results, in particular on the Green-Tao-Ziegler theorem on prime values of linear forms in two variables, on the theory of specialization in arithmetic geometry, and on finite group theory. 
}

\section{Introduction}

This study is motivated by a problem in inverse Galois theory. We first describe the problem and the new results we obtain. Then we discuss the methods that needed to be developed which are of interest by themselves.

\subsection{The Minimal Ramification Problem}
The \emph{inverse Galois problem}, which is one of the central problems in Galois theory, asks whether every finite group $G$ can be realized as  the Galois group $G\cong \Gal(N/\QQ)$ of a Galois extension $N$ of $\QQ$. 
This problem is widely open. There are several different approaches to attack this problem that yield realizations of certain families of groups. The three main approaches found in the literature are:
{\renewcommand{\theenumi}{{\Roman{enumi}}}
\begin{enumerate}
\item \label{specializationapproach} Specializations of geometrically irreducible branched coverings of $\PP^1_{\QQ}$ using Hilbert's irreducibility theorem; see \cite{MalleMatzat,SerreTopics,Voelklein}. 
\item \label{CFTapproach}
Class field theory; see \cite[\S2.1.1]{SerreTopics} or \cite[\S9.6.1]{NSW}.
\item
\label{GRapproach} Galois representations; see \cite[\S5]{SerreTopics} or \cite{KLS1,Wiese,Zywina} for some recent results.
\end{enumerate}
}

The \emph{minimal ramification problem} is a quantitative version of the inverse Galois problem:
For a nontrivial finite group $G$, let $m(G)$ 
be the minimal integer $m$ for which there exists a Galois extension $N/\QQ$ that is ramified at exactly $m$ primes (including the infinite one) such that $\Gal(N/\QQ) \cong G$. 
If no such $N$ exists,  put $m(G)=\infty$. 

The minimal ramification problem asks to calculate or to bound $m(G)$. 
Boston and Markin \cite[Theorem~1.1]{BostonMarkin} prove that if $G\neq 1$ is abelian, then $m(G)=d(G)$, where $d(G)$ is the minimal number of generators of $G$. It is convenient to put $d(1)=1$, and  then since $\QQ$ has no unramified extensions, one gets the lower bound
\begin{equation}\label{eq:BMbnd}
m(G) \geq d(G^{ab}),
\end{equation}
for any nontrivial $G$, where $G^{ab}=G/[G,G]$ is the abelianiztion of $G$. Boston and Markin~\cite{BostonMarkin} conjecture that equality actually holds:

\begin{conjecture}[Boston-Markin]\label{conj:BM}
$m(G)=d(G^{ab})$ for all nontrivial  finite groups $G$.
\end{conjecture}

This conjecture has a lot of evidence in the literature mostly for solvable groups; for example, Jones and Roberts \cite{JonesRoberts} build certain number fields ramified at one prime. 

For solvable groups $G$, one can use Approach~\ref{CFTapproach}, to obtain upper bounds on $m(G)$ and for some subclasses of solvable groups, the full conjecture, see  \cite{BostonMarkin,KisilevskyNeftinSonn,KisilevskySonn,MU2011,Nomura,Plans}. For example, Kisilevsky, Neftin, and Sonn \cite{KisilevskyNeftinSonn} establish the conjecture for semi-abelian $p$-groups. However, to-date, the conjecture is widely open for $p$-groups.

For linear groups, Approach~\ref{GRapproach} is  very effective in giving bounds on ramification. For example, for every prime $p\geq 5$, Zywina \cite{Zywina} realizes ${\rm PSL}_2(\mathbb{F}_p)$ with ramification $\{2,p\}$. (This work is the first realization of these groups as Galois groups \emph{for all} $p$.)

For the special case, $G=S_m$, the symmetric group, the literature contains both theoretical and computational bounds on $m(S_m)$ using Approach~\ref{specializationapproach}: Plans \cite[Remark~3.10]{Plans} remarks that under the deep conjecture in number theory, the Schinzel Hypothesis H, $m(S_m)=1$, as the conjecture predicts; however, an unconditional uniform bound for $m(S_m)$ does not seem to be in the literature. Malle and Roberts \cite{MalleRoberts} construct $S_m$-extensions that are unramified outside at $\{2,3\}$ for some $m$'s between $9$ and $33$.

An analogue of the minimal ramification problem for function fields; that is, when one replaces $\QQ$ by $\FF_q(T)$ is also treated in the literature; see e.g.\ \cite{Witt,Hoelscher}. In this case, it closely relates to the Abhyankar conjecture about the finite quotients of the \'etale fundamental group of an affine curve over an algebraically closed field of positive characteristic that was resolved by Harbater \cite{Harbater} and Raynaud \cite{Raynaud}.

The methods used for non-solvable groups that were discussed above yield a specific extension that realizes the group with a few ramified primes. This is reflected by the fact that proving the conjecture for $G$ and $H$ do not yield a solution for $G\times H$. 
We propose to study the conjecture, in the following asymptotical formulation:
\begin{equation}\label{eq:powers}
m(G^n) = \begin{cases}
d(G^{ab}) \cdot n, & G^{ab}\neq 1\\
1,&G^{ab}=1.
\end{cases}
\end{equation}
To the best of our knowledge, there is no strong evidence for the case of perfect $G$ and large $n$.

In this work we propose an attack on the minimal ramification problem using Approach~\ref{specializationapproach}. Our method produces novel results for groups having a realization as the Galois group of a branched covering and it may be applied to direct products; hence in the asymptotic formulation  \eqref{eq:powers} we get new strong upper bounds, and sometimes asymptotic formulas. The results are discussed in detail below. This attack necessitates developing the theory of specializations, and combining it with sieve theory results on prime values of polynomials, such as combinatorial sieve \cite{HalberstamRichert} and the Green-Tao-Ziegler theorem~\cite{GTZ}.

\subsection{Main Results}
All of our results are for groups $G$ that can be realized as the Galois group of a geometrically irreducible branched covering $\phi \colon C\to \PP^1_{\QQ}$ defined over $\QQ$. 

In particular, for any such group we prove:
\begin{equation}\label{eq:generalG}
m(G^n) = O(n), \qquad n\to \infty,
\end{equation}
where the implied constant is given explicitly. Note that if $G$ is not perfect, then by the simple observation  \eqref{eq:BMbnd} one gets that \eqref{eq:generalG} gives the correct order of magnitude in the sense that 
\[
m(G^n) = \Theta(n).
\]
Further assume that the branch locus of $\phi$ consists on $r$ rational points, then
\begin{equation}\label{eq:generalGrat}
m(G^n) \leq (r-1)n + O(1), \qquad n\to \infty.
\end{equation}
We get a better bound if our group $G$ satisfies the so called $E(p)$-condition for some prime number $p$:  \emph{all} the nontrivial simple quotients of $G$ are $p$-groups, but \emph{none} of the quotients of the commutator $[G,G]$ are (see Definition~\ref{def:Sp} and the examples that follow; e.g., the symmetric group is $E(2)$). 
Assuming $d(G^{ab}) \leq r-2$, we get 
\begin{equation}\label{prime_quotient}
d(G^{ab}) \cdot n \leq m(G^n) \leq (r-2)n + O(1), \qquad n\to \infty.
\end{equation}
In the special case when $G=S_m$, which is of particular interest, we have $r=3$, and we get
\begin{equation}\label{SN-realization}
n\leq m(S_m^n)\leq n+4, \qquad \forall n\geq 1,\ m>0.
\end{equation}
For $n=1$, we can do even better:
\begin{equation}\label{SN-realizationn=1}
m(S_m)\leq 4, \qquad \forall m>0.
\end{equation}
In particular, $m(S_m)$ is bounded.

We emphasize that in  \eqref{SN-realization} and \eqref{SN-realizationn=1} the infinite prime is ramified; that is to say, the minimal number of prime numbers that ramify in $S_m^n$ and $S_m$ extensions is at most $n+3$ and $3$, respectively. 

We note that our bounds in \eqref{SN-realization} and \eqref{SN-realizationn=1} are independent of $m$ and are unconditional. This comes in contrast to the hitherto known results \cite{Plans} that were conditional on the Schinzel Hypothesis H and restricted to $n=1$.

In general, constructing branched covering $\phi\colon C\to \PP^1_{\QQ}$ with specific Galois group $G$ 
is notoriously difficult. The classical method of rigidity, reduces this problem to the group theoretical problem of finding a \emph{rigid} tuple; see  \S\ref{sec_Rigid} or the books \cite{MalleMatzat, SerreTopics,Voelklein}. 
If $G$ has a rational rigid $r$-tuple, then we prove that
\begin{equation}\label{rigid}
m(G)\leq r + \#(\Prms(|G|) \cup \{p\leq r\}).
\end{equation} 
If in addition $G$ satisfies the $E(p)$-condition, then $d((G^n)^{ab}) = d(G^{ab})n$ and we establish the sharp asymptotic formula:
\begin{equation}\label{rigidrigidrigid}
m(G^n) = d(G^{ab})\cdot n + O\left(\frac{n}{\log(n)}\right).
\end{equation}

Finally we  remarks that the methods above work also for general direct products of groups and we have restricted the discussion to direct powers merely for simplicity of presentation. For example, the same proof of \eqref{SN-realization} gives that
\[
m(\prod_{i=1}^{n} S_{m_i})\leq n+4.
\]

\subsection{Methods}
We always write elements of $\PP^1(\QQ)$ as pairs $[a:b]$ with $a,b\in \ZZ$ relatively prime. This presentation is unique up to a sign. For us a prime $p$ denotes either a prime number or the infinite prime of $\QQ$. The completion at $p$ is denoted by $\QQ_p$, so in particular, $\QQ_{\infty}=\RR$.
Every finite set of primes $S$ defines the $S$-adic topology on $\PP^1(\QQ)$ induced by the diagonal embedding $\PP^1(\QQ)\to \prod_{p\in S} \PP^1(\QQ_p)$. 
For a finite set of primes $S$ that contains $\infty$ and for an integer $n\in \ZZ$  we denote
\begin{equation}\label{eq:Prms_S(n)}
\Prms_S(n)=\left\{p : p\mid n \right\}\smallsetminus S .
\end{equation}
The following function plays a key role in the investigation.

\begin{definition}\label{def:B}
Let $D_1, \ldots, D_r\in \ZZ[t,s]$ be non-associate irreducible homogeneous polynomials and $D=\prod_{i} D_i$. We defined $B(D_1, \ldots, D_r)$ to be the minimal positive integer $B$ for which there exists a finite set of primes $S_0=S_0(B)$ that contains $\infty$ such that for every finite set of primes $S_0\subseteq S$ and nonempty $S$-adic neighbourhood $V_S\subseteq \PP^1(\QQ)$ there exists $[a:b]\in V_S$ such that
\[
\#\Prms_S(D(a,b)) \leq B. 
\]
We immediately remark that it follows that there exists infinitely many such $[a:b]$ in each $V_S$. 

For an $r$-tuple  $\mathbf{d} = (d_1, \ldots, d_r) $ of positive integers, we let 
\begin{equation}\label{eq:Bd}
B(\mathbf{d})  = \max_{(D_1,\ldots, D_r)} B(D_1,\ldots, D_r),
\end{equation}
where $(D_1, \ldots , D_r)$ runs over all non-associate irreducible homogenous polynomials of degrees $\deg D_i = d_i$.
\end{definition}

It is far from being obvious that $B(\mathbf{d})$ is finite. However sieve methods may be used to derive  effective bounds in terms of $r$ and $d=\sum_id_i$. From~\cite[Theorem 10.11]{HalberstamRichert} the general bound
\begin{equation}\label{eq:B-gen-bd}
B(\bfd) \leq d - 1+ r\sum_{j=1}^r\frac{1}{j} + r \log\left( \frac{2d}{r} + \frac{1}{r+1} \right) 
\end{equation}
may be derived.
Schinzel Hypothesis H on prime values of polynomials implies
\begin{equation}
\label{eq:B-conj-bd}
B(\bfd) \leq  r. 
\end{equation}
When all $d_i=1$, the Green-Tao-Ziegler theorem~\cite{GTZ} achieves this bound:
\begin{equation}\label{eq:B-GTZ-bd}
 B(1,\ldots, 1)\leq r.
\end{equation}
The formal derivations of all of these results appears in \S\ref{sec:Schinzel}.

Another key notion in our results is that of universally ramified primes:
Let $\phi\colon C\to \PP^1_{\QQ}$ be a geometrically irreducible branched covering.  For each point $[a:b]\in \PP^{1}(\QQ)$, we let $A^{\phi}_{[a:b]}$ be the specialized  algebra at $[a:b]$ which is defined by
\[
\phi^{-1}([a:b])=\Spec(A^{\phi}_{[a:b]}).
\]
Note that $A^{\phi}_{[a:b]}$ is a finite $\QQ$-algebra of degree $[A^{\phi}_{[a:b]}:\QQ]=\deg \phi$ and it is \'etale resp.\ a field if and only if $[a:b]$ is not a branch point of $\phi$ resp.\ $\phi^{-1}([a:b])$ is $\QQ$-irreducible. The set of \emph{universally ramified primes} is defined as
\[
U=U(\phi)= \bigcap_{[a:b]\in \PP^1(\QQ)} \Ram(A^{\phi}_{[a:b]}/\QQ),
\]
where for a finite $\QQ$-algebra $A$ we let $\Ram(A/\QQ) = \{ p  \mid A \otimes \QQ_{p}^{ur} \not \cong (\QQ_p^{ur})^n\}$. 
Here $\QQ_{\infty}^{ur}=\RR$. We also write
\[
\Ram_S(A/\QQ) = \Ram(A/\QQ)\smallsetminus S,
\]
where $S$ is a finite set of primes. We note that $p\not \in \Ram(A/\QQ)$ if and only if $A$ is isomorphic to a product of number fields that are unramified at $p$. Thus $U$ is the set of the primes that ramify under every specialization.
In practice it is easy to bound $U$ from above, simply by taking some random points $[a:b]\in \PP^1(\QQ)$ and calculating the greatest common divisor of the discriminants of the specialized algebras $A^{\phi}_{[a:b]}/\QQ$. However, to calculate $U$ exactly, may be  difficult. 

We denote by $\RamG(\phi)\subset \PP^1_\QQ$ the closed subscheme of branch points of $\phi$. So $\RamG(\phi)$ is the zero locus of some nonzero homogenous polynomial $D(t,s)\in \ZZ[t,s]$.

The last notion we need in order to state the main tool we develop in this paper, is of thin sets~\cite{SerreTopics} in the sense of Serre: A thin set of type 1 in $\PP^1(\QQ)$ is a finite set. A thin set of type 2 is $\phi(C(\mathbb{Q}))$, where $\phi\colon C\to \PP^1_\QQ$ is an irreducible branched covering of degree $\geq 2$. A \textbf{thin} set in $\PP^1(\QQ)$ is a set contained in a finite union of thin sets of types 1 and 2. So the Hilbert irreducible theorem is the statement that $\PP^1(\QQ)$ is not thin.

\begin{theorem}\label{thm:main}
Let $\phi\colon C\to \PP^1_{\QQ}$ be a geometrically irreducible branched covering.
Let $U=U(\phi)$ be the set of universally ramified primes and  $\RamG(\phi) = \{(D_1), \ldots, (D_r)\}\subseteq \PP^1_\QQ$ be the branch locus of $\phi$, where $D_i\in\ZZ[t,s]$ are non-associate homogeneous irreducible polynomials.
Then the set $\Omega$ of all  $[a:b]\in \PP^1(\QQ)$ such that
$\#\Ram_U(A^{\phi}_{[a:b]}/\QQ) \leq B(D_1,\ldots, D_r)$ is not thin.
\end{theorem}

Theorem~\ref{thm:main} follows from a strong version of Hilbert's irreducibility theorem and the following result on ramification under specialization.

\begin{theorem}\label{thm:proj}\label{thm:Ram-Spec}
Under the notation of Theorem~\ref{thm:main} and with $D=D_1\cdots D_r$ there
exists a finite set of primes $T_\phi$ containing $U\cup\{\infty\}$ such that for every finite set of primes $S$ with $T_\phi\subseteq S$ there exists a nonempty $S$-adic open set $V_S$ of $\PP^1(\QQ)$ satisfying the following property: For every $[a:b] \in V_S$ we have
\begin{enumerate}
\item $\Ram(A^{\phi}_{[a:b]}/\QQ)\cap S  = U$.
\item $\Ram_S(A^{\phi}_{[a:b]}/\QQ) \subseteq \Prms_S(D(a,b))$.
\end{enumerate}
\end{theorem}

\section{Proof that Theorem~\ref{thm:Ram-Spec} implies   Theorem~\ref{thm:main}}
Hilbert's irreducibility theorem states that $\PP^1(\QQ)$ is not thin. We shall need a strong variant of the theorem that gives $S'$-adic neighbourhoods in the complement of any thin set:

\begin{lemma}\label{theorem}\label{thm:HIT}
Let $Z$ be a thin set in $\PP^1(\QQ)$.
For every finite set of primes $S$ there exists a finite set of primes $S'$ and a nonempty $S'$-adic neighbourhood $V_{S'}$ such that $S\cap S'=\emptyset$ and $Z\cap V_{S'}=\emptyset$.
\end{lemma}

\begin{proof}
By~\cite[Theorem~3.5.3]{SerreTopics} there exists $S'$ with $S'\cap S=\emptyset$ such that $Z$ is not $S'$-adic dense in $\prod_{p\in S'} \PP^1(\QQ_p)$.
So there exists an open subset $U$ of $\prod_{p\in S'} \PP^1(\QQ_p)$ with $Z\cap U=\emptyset$. Since $\PP^1_{\QQ}$ has the weak approximation property (Page~30 in \emph{loc.cit.})\ $V_{S'}:=U\cap \PP^1(\QQ)\neq \emptyset$, as needed.
\end{proof}

\begin{proof}[Proof of Theorem~\ref{thm:main}]
It suffices to show that $\Omega$ is not contained in any thin set $Z$.
Put $B=B(D_1,\ldots, D_r)$, and let $T_\phi$ be as in  Theorem~\ref{thm:Ram-Spec}. Let $S_0=S_0(B)$ be the set of primes from Definition~\ref{def:B}. By Theorem~\ref{thm:Ram-Spec}, for $S_1=T_{\phi}\cup S_0$, there exists a nonempty $S_1$-adic neighbourhood $V_{S_1}$
such that for every $\zeta=[a:b]\in V_{S_1}$ we have
\begin{equation}\label{eq:pf-main1}
\#\Ram_U(A^{\phi}_\zeta) = \#\Ram_{S_1}(A^{\phi}_\zeta) \leq \#\Prms_{S_1}(D(a,b)).
\end{equation}

By Lemma~\ref{thm:HIT} there exists a finite set of primes  $S_1'$ such that  $S_1\cap S_1'=\emptyset$ and there exists a nonempty $S_1'$-adic neighbourhood $V_{S_1'}$ such that $V_{S_1'}\cap Z=\emptyset$. Thus $V_S=V_{S_1}\cap V_{S_1'}$ is an $S$-adic neighbourhood, for $S=S_1\cup S_{1}'$ which is nonempty by the Chinese Remainder Theorem and that satisfies 
\begin{equation}\label{eq:ES}
V_S\cap Z=\emptyset.
\end{equation}   
Since $S_0\subseteq S$, by Definition~\ref{def:B} and by \eqref{eq:pf-main1}, there exists $\zeta\in V_S$ such that 
\[
\#\Ram_U(A^{\phi}_\zeta) \leq B.
\]
This together with \eqref{eq:ES} implies that $\zeta\in \Omega\smallsetminus Z$, so $\Omega\not\subseteq Z$.
\end{proof}

\section{Ramifications}
\subsection{Preliminaries in Commutative Algebra}
Recall that by $[a:b]\in \PP^1(\QQ)$, we always mean that $a$ and $b$ are co-prime integers. This uniquely defines the pair $a,b$, up to a sign.
Given an homogenous ideal $I \lhd \ZZ[t,s]$ and $[a:b] \in \PP^1(\QQ)$, we denote by
\[
I([a:b]) = \{f(a,b) : f\in I\} \lhd \ZZ
\]
which is an ideal in $\ZZ$.
For a prime number $p$ we denote by
$v_p(n)$ the $p$-adic valuation of $n$. We extend the functions $v_p(\bullet)$ and $\Prms_S(\bullet)$ (defined in \eqref{eq:Prms_S(n)}) from the integers to ideals
in the obvious way:
If $J= (n) \lhd \ZZ$, then
\begin{align*}
\Prms_S(J)&=\Prms_S(n) , \quad \mbox{and}\\
v_p(J)&=v_p(n).
\end{align*}
For a prime number $p$, recall that $\QQ_p^{ur}$ denotes the maximal unramified extension of $\QQ_p$  and that $\ZZ_p^{ur}$ is the integral closure of $\ZZ_p$ in $\QQ_p^{ur}$; i.e., the subring of elements with non-negative valuation (w.r.t.\ the unique lifting of $v_p$ to $\QQ_p^{ur}$).

\begin{lemma}\label{lem:finitesetofprimes_proj}
Let $I\lhd \ZZ[t,s]$ be a nonzero homogeneous ideal and let $D(t,s) \in \ZZ[t,s]$ be a homogeneous polynomial such that $D\QQ[t,s]=I\QQ[t,s]$. Then, there exists a finite set of primes $S$ that contains the infinite prime such that for every $[a:b]\in \PP^1(\QQ)$ 
and for every $p\not\in S$ we have $v_p(I(a,b)) = v_p(D(a,b))$.
\end{lemma}

\begin{proof}
The ideal $I$ is generated by finitely many homogeneous polynomials, say $I=\sum_{i=1}^k g_i \ZZ[t,s]$. Thus $D\QQ[t,s] = \sum_{i=1}^k g_i \QQ[t,s]$, which implies that there exist homogeneous polynomials 
\[
c_1(t,s), \ldots, c_k(t,s), d_1(t,s), \ldots, d_k(t,s)\in \QQ[t,s]
\] 
such that $g_i=c_i D$, $i=1,\ldots, k$ and $D= \sum_{i=1}^k d_i g_i$. Let $S'$ be the set of primes dividing the denominators of the coefficients of  $c_1, \ldots, c_k,d_1, \ldots, d_k$. Let $S:= S' \cup \{\infty\}$. Then, for $[a:b]\in \PP^1(\QQ)$ and $p\notin S$, we have that $c_i, d_i \in \ZZ_p[t,s]$ for all $1\leq i \leq k$;
thus $I\ZZ_p[t,s]  = D\ZZ_p[t,s]$. For every $[a:b] \in \PP^1(\QQ)$ we thus have $\ZZ_p I(a,b)= \ZZ_p (D(a,b))$, hence the desired assertion.
\end{proof}

\begin{lemma}\label{lem:hensel}
Let $p$ be a finite prime and let 
\[
\phi\colon F \to \Spec{\ZZ_p^{ur}}
\]
be an \'etale map of degree $n$.
Then $F \cong \Spec{(\ZZ_p^{ur})^n}$. 
\end{lemma}
\begin{proof}
The ring $\ZZ_p^{ur}$ is a Henselian ring with algebraically closed residue field. Thus the assertion follows from \cite[Proposition~I.4.4]{Milne}.
 \end{proof}

Let $\phi\colon C \to \PP^1_\QQ$ be a branched covering. The branch locus $\RamG(\phi)\subset \PP^1_\QQ$ is a closed subscheme of dimension $0$, so $\RamG(\phi)$ is the zero locus of some nonzero homogenous polynomial $D(t,s)\in \ZZ[t,s]$.

The following fact on the closeness of the branch locus over $\ZZ$ is well known.
\begin{lemma}
Let $\phi_{\ZZ}\colon\mathfrak{C} \to \PP^1_\ZZ$ be the normalization of $\PP^1_\ZZ$ in the generic point of $C$ via $\phi$. Then, the branch locus $\mathfrak{R}_{\phi} \subset \PP^1_\ZZ$ of $\phi_{\ZZ}$ is closed.
\end{lemma}

\begin{proof}
Since  $\PP^1_\ZZ$ is Nagata \cite[Tag 035B]{stacks-project}, hence universally Japanese \cite[Tag 033Z]{stacks-project}, and since $\PP^1_\ZZ$ is integral, it is Japanese which means by definition that $\phi_\ZZ$ is finite. 
Thus, by \cite[Tag 024P]{stacks-project}, the ramification locus consists of all $x$ at which the stalk of the coherent sheaf $\Omega_{\mathfrak{C}/\PP^1_{\ZZ}}$ is nontrivial. The sheaf $\Omega_{\mathfrak{C}/\PP^1_{\ZZ}}$ is locally of finite type by \cite[Tag 01V2]{stacks-project}  hence  \cite[Tag 01BA]{stacks-project} implies that the ramification locus is closed. Thus we conclude that the branch locus $\mathfrak{R}_{\phi}$, which is the image of the ramification locus under $\phi_{\ZZ}$ is  closed in $\PP^1_{\ZZ}$ as finite morphisms are closed. 
\end{proof}

Away from $\mathfrak{R}_{\phi}$, the morphism $\phi_{\ZZ}$ is \'etale.
We denote by $d_{\phi,\ZZ} \lhd \ZZ[t,s]$ the homogenous ideal  that defines $\mathfrak{R}_{\phi}$.
We have:
\begin{equation}\label{eq:branchideals}
d_{\phi,\ZZ} \QQ[t,s] = D(t,s)\QQ[t,s].
\end{equation}
\begin{proposition}\label{prop:disc_ram}
Let $\phi\colon C \to \PP^1_\QQ$ be a branched covering, let $[a:b] \in \PP^1(\QQ)$, and let $p$ be a prime number. Assume that $v_p(d_{\phi,\ZZ}(a,b))=0$. Then
$A^{\phi}_{[a:b]}$ is unramified at $p$.
\end{proposition}
\begin{proof}
We identify $\PP^1_{\FF_p}=\PP^1_{\ZZ}\times_{\ZZ} \FF_p$ and $\PP^1_{\QQ} = \PP^1_{\ZZ}\times_{\ZZ} \QQ$ as subschemes of $\PP^1_\ZZ$: the \emph{special} and the \emph{generic} fibers, respectively. 
Write $\zeta=[a:b]\in \PP^1_{\QQ}$ and $\zeta_p=[\bar{a}:\bar{b}]\in \PP^1_{\FF_p}$, where the over-line denotes reduction modulo $p$. 
By assumption, there exists $f \in d_{\phi,\ZZ}$ such that $f(a,b) \not \equiv 0 \mod p$, which implies that
\[
\zeta_p \not\in \mathfrak{R}_{\phi}.
\]
So $\phi_\ZZ$ is \'etale at $\zeta_p$.
We base change with $\ZZ_p^{ur}$ to get the following diagram:
\[
\xymatrix{
  F_{\zeta}\ar[d]\ar[r] & \mathfrak{C}_{\ZZ^{ur}_p}\ar[d]^{\phi_{\ZZ^{ur}_p}}\ar[r]&\mathfrak{C}\ar[d]^{\phi_\ZZ}\\
  \Spec \ZZ^{ur}_p \ar[r]^{\zeta} & \PP^1_{\ZZ^{ur}_p}\ar[r]&\PP^1_{\ZZ}
 }
\]
Since $\phi_{\ZZ_p^{ur}}$ is \'etale in a neighborhood of $\zeta_p$,
the fiber   $F_{\zeta}$ is \'etale over $\Spec\ZZ_p^{ur}$. By Lemma~\ref{lem:hensel},
\[
F_{\zeta} \cong \Spec {(\ZZ_p^{ur})}^{\deg \phi};
\]
so
\[
\Spec({A^{\phi}_{\zeta}} \otimes_{\ZZ} \QQ_{p}^{ur}) = F_{\zeta} \times_{\Spec(\mathbb{Z}_p^{ur})} \Spec(\QQ_p^{ur}) = \Spec (\QQ_p^{ur})^{\deg \phi}.
\]
This implies that $A^{\phi}_{\zeta}$ is unramified at $p$, as needed. 
\end{proof}

\begin{lemma}\label{lem:open-unr}
Let $\phi\colon C\to \PP^1_\QQ$ be a branched covering, and let $S$ be a finite set of primes. Then there exists a nonempty $S$-adic open set $V_S$ of $\PP^1(\QQ)$ such that for every $\zeta\in V_S$ we have $\Ram(A^{\phi}_\zeta/\QQ)\cap S \subseteq U(\phi)$.
\end{lemma}
\begin{proof}
By the Chinese Reminder Theorem, if $S_1\cap S_2=\emptyset$ and if $V_{S_i}$ is a  nonempty open $S_i$-adic open set, $i=1,2$, then $V_{S_1}\cap V_{S_2}$ is a nonempty $S_1\cup S_2$-adic  set.  Thus it suffices to consider the case where $S = \{p\}$; i.e., $S$ contains only one prime. 

If $p \in U(\phi)$, then the assertion is trivial. Otherwise, there exists $\zeta \in \PP^1(\QQ)$ such that $A^{\phi}_{\zeta}$ is unramified at $p$. In particular,   $A^{\phi}_\zeta$ is reduced, so $\zeta \not \in \RamG(\phi)$.
Consider the map $$\phi_{\QQ_p^{ur}}\colon C(\QQ_p^{ur}) \to \PP^1(\QQ_p^{ur}).$$ Since $\zeta$ is not a branch point, $\#\phi_{\QQ_p^{ur}}^{-1}(\zeta) (\QQ_p^{ur})= \deg \phi$, and so as $\QQ_p^{ur}$ is Henselian, by the inverse function theorem (see e.g.\ \cite[Corollary 9.5]{GPR}) there exists some  $p$-adic neighbourhood  $V$ of $\zeta$ such that for every $\zeta' \in V$
\[
 \# \phi_{\QQ_p^{ur}}^{-1}(\zeta') (\QQ_p^{ur}) = \#\phi_{\QQ_p^{ur}}^{-1}(\zeta) (\QQ_p^{ur})= \deg \phi .
\]
The proof is done with $V_S=V\cap\PP^1(\QQ)$.
\end{proof}

\section{Proof of Theorem~\ref{thm:proj}}\label{ssec:prf_prj}
Let 
$I=d_{\phi,\ZZ} \lhd \ZZ[t,s]$. Since $D\QQ[t,s] = I\QQ[t,s]$, by Lemma~\ref{lem:finitesetofprimes_proj} there exists a finite set of primes $S_1$ such that for all $p\not\in S_1$ and for all $[a:b]\in \PP^1(\QQ)$ we have
\begin{equation}\label{eq:pj-samevalue}
v_p(I(a,b)) = v_p(D(a,b)).
\end{equation}
Let $T_\phi=
S_1\cup U \cup \{\infty\}$ and let $S$ be a finite set of primes containing $T_\phi$.
By Lemma~\ref{lem:open-unr}, there exists a nonempty $S$-adic open set $V_S$ of $\PP^1(\QQ)$ such that for all $\zeta\in V_S$ we have $\Ram(A^{\phi}_\zeta/\QQ) \cap S\subseteq U$, so $\Ram(A^{\phi}_\zeta/\QQ) \cap S=U$.

Let $\zeta =[a:b] \in V$, let $p\not\in S$, hence $p\not\in S_1$, and assume that $p\nmid D(a,b) $.
 By \eqref{eq:pj-samevalue}, we have $v_p(I(a,b)) = v_p(D(a,b)) =0$, so $p$ is prime to $I(a,b)$.
This implies, by Proposition~\ref{prop:disc_ram}, that $p \not \in \Ram(A^{\phi}_\zeta 
/\QQ)$ \qed.

\section{Prime Values of Polynomials}
\label{sec:Schinzel}
The goal of this section is to formally deduce \eqref{eq:B-gen-bd} and \eqref{eq:B-GTZ-bd}
from sieve theoretical results and \eqref{eq:B-conj-bd} conditionally on 
Schinzel Hypothesis H.

\subsection{Local Obstructions}
Since many of the results in this theory are stated in the literature for univariate polynomials we first deals with those, and then move to bivariate homogeneous polynomials.

We say that $f(x)\in \ZZ[x]$ has a \emph{local obstruction at $p$} if $p$ divides $f(n)$ for all $n\in \ZZ$. We denote the set of  primes at which there is a local obstruction by $O_f$.

\begin{lemma}\label{lem:obstructionimpliessmallprime}
If $f$ is primitive (i.e.\ the greatest common divisor of its coefficients is $1$), then $p\leq \deg f$ for all $p\in O_f$.
\end{lemma}

\begin{proof}
By assumption $f\mod p\in \FF_p[x]$ is not the zero polynomial, hence has at most $\deg f$ roots modulo $p$.
\end{proof}

\begin{definition}\label{B_0(d)}
Let $d_1, \ldots, d_r$ be positive integers. Define 
\[
B_0=B_0(d_1, \ldots, d_r)
\] 
to be the minimum positive integer $B_0$ such that for every $f=f_1\cdots f_r$, with $f_i(x)\in \ZZ[x]$ irreducible of degree $d_i$, with positive leading coefficient, and with $O_{f}=\emptyset$ there exist infinitely many $n>0$ such that $\#\Prms(f(n))\leq B_0$.
\end{definition}

Sieve methods are effective in bounding $B_0$ in
 terms of $r$ and $d=\sum_{i=1}^r d_i$: By the beta-sieve, \cite[Theorem 10.11]{HalberstamRichert} we have 
\begin{equation}\label{eq:B-gen-bd-}
B_0(d_1,\ldots, d_r) \leq b
\end{equation}
for every
\[
b> d -1 + r\sum_{j=1}^r\frac{1}{j} + r \log\left( \frac{2d}{r} + \frac{1}{r+1} \right).
\]
Schinzel Hypothesis H is a more precise conjecture that says that
\[
B_0(d_1, \ldots, d_r) \leq r.
\]
(Note that one cannot do better.)
Hence to obtain \eqref{eq:B-gen-bd} and \eqref{eq:B-conj-bd} it suffices to prove that
\begin{equation}\label{eq:BvsB_0}
B(\bfd) \leq B_0(\bfd),
\end{equation}
which we now pursue. First we remove the restriction of the having no local obstructions:

\begin{lemma}\label{lem:Schinzel_local_obstruction}
Let $f_1, \ldots, f_r\in \ZZ[x]$ be irreducible polynomials of positive leading coefficients and of respective degrees $d_1, \ldots, d_r$, $f=f_1\cdots f_r$, and  $S$ a finite  set of primes such that $f$ has no local obstructions outside of $S$. Then, there exists infinitely many $n$ such that $\#\Prms_S(f(n))\leq B_0(d_1,\ldots, d_r)$.
\end{lemma}

\begin{proof}
For each $p\in S$ let $\alpha_p$ be the maximal non-negative integer such that the function $n\mapsto f(n)\mod p^{\alpha_p}$ is the zero function. 
Put $N=\prod_{p} p^{\alpha_p}$ and choose an integer $a_p$ such that $f(a_p)\not \equiv 0 \mod{p^{\alpha_p+1}}$.  
By the Chinese Reminder Theorem, we have an integer $a$ with $a\equiv a_p\pmod{p^{\alpha_p+1}}$ for all $p\in S$ and let  $g(y) = \frac{f(N y + a)}{N}$.

We claim that $g(y)$ is an integral polynomial with no local obstructions. Indeed, 
since $(x-a)$ divides $f(x)-f(a)$ in $\ZZ[x]$ we get, by substitution $x=Ny+a$, that $Ny$ divides $f(Ny+a)-f(a)$ in $\ZZ[y]$. Since $N\mid f(a)$, $N$ divides the coefficients of $f(Ny+a)=(f(Ny+a)-f(a))+f(a)$, so $g(y)\in \ZZ[y]$. To show that $g(y)$ has no local obstruction at a prime $p$, we note that if $p\in S$, then $g(0) \not \equiv 0 \mod p$ 
and if $p\not\in S$, then $f$ does not have local obstruction at $p$, hence there exists $m$ with $f(m)\not\equiv 0\pmod{p}$, and since $p\nmid N$, there is $n$ such that $m\equiv Nn+a \pmod p$, hence $g(n)\not\equiv 0\pmod{p}$.

Next we apply the definition of $B_0=B_0(d_1, \ldots, d_r)$ to $g$ (which has the same factorization type as $f$) and the trivial observation that $\Prms_{S}(f(Nn+a))= \Prms_S(g(n))$ to conclude that for infinitely many $n$ we have
\[
\#\Prms_{S}(f(Nn+a))\leq \#\Prms(g(n))\leq B_0,
\]
as needed.
\end{proof}

Let $N$ be a positive integer and $S:= \Prms(N) \cup \{\infty\}$  we define  $V_N$ to be the following $S$-adic neighborhood of $[1:0] \in \PP^1(\QQ)$:
\begin{equation}\label{eq:standV_N}
V_N := \left\{ [a:bN]\in \PP^1(\QQ) : a,b\in \ZZ \mbox{ and } \left|\frac{bN}{a}\right|\leq \frac{1}{N}\right\}.
\end{equation}
Note that by our notational agreement, $\gcd(a,bN)=1$.

\begin{lemma}\label{lem:uni-to-bi}
For every $D=D_1\cdots D_r$ with $D_1, \ldots, D_r\in \ZZ[t,s]$  homogeneous irreducible polynomials of respective positive degrees $d_1, \ldots, d_r$, there exists a finite set of primes $S_0=S_0(d_1, \ldots, d_r)$ depending only on $d_1,\ldots, d_r$ such that for every positive integer $N$ there exists $[a:b]\in V_N$ such that
\[
\#\Prms_{S}(D(a,b))\leq B_0(d_1,\ldots, d_r),\qquad S=S_0\cup\Prms(N).
\] 
\end{lemma}

\begin{proof}
Let $S_0$ be the set of all primes $p$ such that $p\leq \deg D$.
If $p\nmid N$, then
\[
\{[1+xN :N] \in \PP^1(\FF_p) \mid x\in \FF_p\}=\AA^{1}(\FF_p).
\]
Thus if $D(1+xN,N) \equiv 0 \pmod p$ for all $x$, then $p\in S_0$ by Lemma~\ref{lem:obstructionimpliessmallprime} (note that $D$ is primitive as  the  product of irreducible polynomials in $\ZZ[t,s]$). Therefore for $p\not\in S$, the function $n\mapsto D(1+nN,N) \pmod p$ is nonzero.

Denote $g_i(x) = D_i(1+xN,N)$. If $D_i(t,s)\neq s$, then $g_i$ is an irreducible polynomial of degree $d_i$ in $\QQ[x]$. Moreover, we may write $g_i(x) = c_i f_i(x)$, where $c_i\in \ZZ$ and $f_i(x)\in\ZZ[x] $ is irreducible. By the above $\Prms(c_i) \subseteq S$. If $D_i(t,s) = s$, we denote $f_i(x)=x$.

Now $f_1, \ldots, f_r$ are irreducible in $\ZZ[x]$, $f=f_1\ldots f_r$ has no local obstruction outside of $S$, and $\deg f_i = \deg D_i$. By Lemma~\ref{lem:Schinzel_local_obstruction}, there exists $n\geq N$ such that $\#\Prms_S(f(n)) \leq B_0(d_1, \ldots, d_r)$. This finishes the proof since $\frac{N}{1+nN}<\frac{1}{N}$, so $[1+nN:N]\in V_N$.
\end{proof}

Note that $\GL_2(\ZZ)$ acts transitively on $\PP^1(\QQ)$ by
\begin{equation}\label{actionGL}
\left(\begin{matrix}x_1&x_2\\ y_1&y_2\end{matrix}\right)[a:b]=[x_1a+x_2b : y_1a+y_2b].
\end{equation}

\begin{lemma}\label{lem:transitiveGL2}
Let $S$ be a finite set of primes containing the infinite prime and let $V_S$ be a nonempty $S$-adic neighbourhood. Then, there exist a positive integer $N$ and a matrix $g\in \GL_2(\ZZ)$ such that $gV_N\subseteq V_S$.
\end{lemma}

\begin{proof}
Let $[a:b]\in V_S$ and choose $g\in \GL_n(\ZZ)$ such that $g[1:0]=[a:b]$. Then $g^{-1}(V_S)$ is a neighbourhood of $[1:0]$. Hence there exists $N$ with $\Prms(N)\subseteq S$ such that $V_N\subseteq g^{-1}V_S$, so $gV_N \subseteq V_S$.
\end{proof}

\begin{proof}[Proof of~\eqref{eq:BvsB_0}] Let $D_1,\ldots, D_r\in \ZZ[t,s]$ be non-associate irreducible homogeneous polynomials of positive degrees $d_1,\ldots, d_r$. Let $S_0=S_0(d_1, \ldots , d_r)$ be as in Lemma~\ref{lem:uni-to-bi}.  Let $S$ be a finite set of primes containing $S_0$ and $V_S$ a nonempty $S$-adic neighbourhood. By Lemma~\ref{lem:transitiveGL2}, there exists $N$ with $\Prms(N)\subseteq S$ and $g\in \GL_2(\ZZ)$ such that $gV_N \subseteq V_S$. We let $D'_i = D_i\circ g$ and $D'=D_1'\cdots D_r'$. Then each $D'_i$ is irreducible of degree $d_i$.
By Lemma~\ref{lem:uni-to-bi}, there exists $[a':b']\in V_N$ with $\#\Prms_S (D'(a',b'))\leq B_0(d_1, \ldots, d_r)$ (note that $S=S\cup \Prms(N)$). Therefore, for $[a:b]=g[a':b']$ we get that  $\#\Prms_S(D(a,b))\leq B_0(d_1, \ldots, d_r)$, which proves~\eqref{eq:BvsB_0} by the definition of $B$.
\end{proof}

Equation~\eqref{eq:B-GTZ-bd} immediately follows from the following form of \cite[Corollary~1.9]{GTZ} (which essentially appears in 
Proposition~\cite[Proposition~1.2]{HSW}).
\begin{proposition}
Let $L_i(s,t)=\beta_i t-\alpha_i s$ be distinct primitive integral linear forms, $i=1,
\ldots, r$. Let $S$ be a finite set of primes containing all primes $p\leq r$ and let $V_S$ be a nonempty $S$-adic neighbourhood. Then there exists $[a:b]\in V_S$ such that for all $i=1,\ldots, r$ the value $L_i(a,b)$ is either 
a prime or a unit
 in $\ZZ[S^{-1}]$.
\end{proposition}
\begin{proof}
As $r=1$ follows from Dirichlet's theorem on primes in arithmetic progressions, we may assume w.l.o.g.\ that $r\geq 2$. 
By Lemma~\ref{lem:transitiveGL2}, it suffices to show the following assertion:\\
Let $L_i(s,t)=\beta_i t-\alpha_i s$ be distinct primitive integral linear forms, $i=1,
\ldots, r$. Let $S_0$ be the set of primes  $p\leq r$.
Then, for  every positive integer $N$ there exists $[a:b]\in V_N$ such that
$\#\Prms_{S}(L_i(a,b))\leq 1$, for all $1\leq i\leq r$,  with $S=S_0\cup\Prms(N)$.

Let $N$ be a positive integer and $S=S_0\cup \Prms(N)$. For every $p \in S$, we let
\[
a_p := \max_{i,\beta_i \neq 0}v_p(\beta_i) 
\]
and 
\[
C:= \prod_{p\in S} p^{a_p+1}.
\]
For every $1\leq i\leq r$, we
set $c_i := 
\gcd(
\beta_i,C)$ and 
\[
M_i(t,s) =  \frac{\beta_i t - \alpha_i C N s}{c_i}
\]
if $L_i(t,s) \neq \pm s$, and 
\[
M_i(t,s) = s
\]
if $L_i(t,s) = \pm s$.

We claim that there are no local obstructions; namely, for every prime $p$ there exists $[a:b] \in \PP^1(\QQ)$ such that for all $1\leq i\leq r$ we have $p\nmid M_i(a,b)$. Indeed, if $p \not \in S$, then $\deg \prod M_i= r < p+1$, so such $[a:b]$ exists. 
Otherwise, we take $a = b=1$.

Let 
\[
K:=\left\{(x,y) \in \mathbb{R}^2  \left| 0 < y < \frac{x}{CN^2}  \right. \right\}.
\]
The convex set $K$  and the linear forms $M_i(t,s)$ satisfy the conditions of a theorem of Green-Tao-Ziegler \cite[Cor~1.9]{GT10}\footnote{This theorem is stated in \cite[Cor~1.9]{GT10} conditionally on two conjectures one of which is proved in \cite{GT12} and the other in \cite{GTZ}.} (after replacing $M_i$ by $-M_i$ is necessary).
So, we have infinitely many  $(a,b)\in \ZZ^2 \cap K$ such that
$M_i(a,b)$ is prime for every $1 \leq i \leq r$. Since $S$ is finite, we may choose $(a,b)$ such that $M_{i}(a,b)$ is also not in $S$. This implies that $a$ has no prime factors from $S$. Thus $\gcd(a,N)=1$.
Let $\gamma = \gcd(a,NCb)=\gcd(a,Cb)$. So 
\[
[a/\gamma:N C b/\gamma] \in V_N.
\]
Note that 
\[
L_i(a/\gamma,N C b/\gamma) = c_i/\gamma M_i(a,b).
\]
As $c_i$ is a unit is $\ZZ[S^{-1}]$ and $M_{i}(a,b)$ is a prime in $\ZZ[S^{-1}]$, we get that $L_i(a/\gamma,NCb/\gamma)$ divides a primes and so either a prime or a unit in $\ZZ[S^{-1}]$.
\end{proof}

\section{Universally Ramified  Primes}
Recall that we view an element $g$ of $\GL_2(\QQ)$ as an automorphism $g\colon \PP^1_\QQ\to \PP^1_\QQ$ via the action \eqref{actionGL}.
Given $\phi\colon C\to \PP^1_{\QQ}$ denote by $\phi^{g}\colon C \to \PP^1_{\QQ}$ to composition $g\circ \phi$.
If $\phi$ is generically Galois, then so is $\phi^g$, and
\[
\Gal(\phi) \cong \Gal(\phi^g).
\]
From its definition, the set of universally ramified primes is stable under the action of $g$, that is 
\[
U(\phi^g) = U(\phi).
\]
However, the branch locus is not invariant:
\begin{equation}\label{branchpoints}
\RamG(\phi^g) = g \cdot \RamG(\phi).
\end{equation}
We set
\[
U_{\infty}(\phi) = U(\phi) \smallsetminus \{\infty \}.
\]

For the applications to the minimal ramification problem, we are especially interested in controlling the universally ramified primes in fiber products.
For an element $x \in \mathbb{Q}^{\times}$, we let  $g_{x_{\times}}\in \GL_2(\QQ)$ be the matrix 
\[
g_{x_{\times}}:= \left(\begin{matrix}x&0\\ 0&1\end{matrix}\right) \quad \mbox{and} \quad \phi^{x_{\times}}:=\phi^{g_{x_{\times}}} 
\]
the composition map.
For an element $b \in \mathbb{Q}$, we  let $g_{b^{+}}$ be the matrix 
\[
g_{b^+} :=\left(\begin{matrix}1&b\\ 0&1\end{matrix}\right) \quad \mbox{and}\quad  \phi^{b_{+}}:= \phi^{g_{b^+}}
\]
the composition map.

Recall that if $\phi\colon C\to \PP^1_\QQ$ is a morphism of smooth geometrically connected projective $\QQ$-curves, then  $\phi_\ZZ \colon \frak{C} \to \PP^1_\ZZ$ is the normalization of $\PP^1_\ZZ$ in $C$, and the branch locus $\frak{R}_{\phi}$ is a closed subscheme of $\PP^1_\ZZ$.

We say that a prime number $p$  is \emph{vertically ramified} in $\phi$ if
\[
\PP^1_{\mathbb{F}_p} \subset \frak{R}_\phi \subset \PP^1_\ZZ
\]
(under the natural embedding induced from $\ZZ\to \FF_p$).
This notion is consistent with the one in \cite{Legrand2013}.
We denote the set of vertically ramified primes by $V(\phi)$.
Let $g\in \GL_2(\ZZ_p)$. As an automorphism of $\PP^1_{\ZZ_p}$, it follows that
\begin{equation}\label{eq:Vinvariant}
p\in V(\phi) \Longleftrightarrow p\in V(\phi^g), \qquad \mbox{if} \quad g\in \GL_2(\ZZ_p)\cap \GL_2(\QQ).
\end{equation}
However, for general $g\in \GL_2(\QQ)$ it may happen that  $V(\phi)\neq V(\phi^g)$. We also note that by Abhyankar's lemma, for $\phi_i\colon C_i\to \PP^1(\QQ)$, $i=1,2$, we have
\begin{equation}\label{eq:abhy}
V(\phi_1\times_{\PP^1} \phi_2)  = V(\phi_1) \cup V(\phi_2).
\end{equation}
Let $\RamG(\phi) = \{ (D_1) , \ldots, (D_r)\}$ and let $B(\phi)$ the set of  prime numbers $p$ for which for every $[a:b]\in \PP^1(\FF_p)$ there is $i$ with  $D_i(a,b)=0$ (in $\FF_p$). 
As in Lemma~\ref{lem:obstructionimpliessmallprime}, one has
\begin{equation}\label{eq:Binvariant}
p\in B(\phi) \Rightarrow p+1\leq  \deg(\RamG(\phi)) :=\sum_{i=1}^r\deg D_i.
\end{equation}
In general, we conclude
\begin{equation}\label{eq:ubdbyvb}
 U_{\infty}(\phi) \subseteq V(\phi)\cup B(\phi),
\end{equation}
see~\cite[Specialization Inertia Theorem (1)]{Legrand2013}.

\begin{lemma}\label{lem:uram-1}
Let $\phi \colon C \to \PP^1$ be  a branched covering and let $p\neq q$ be prime numbers such that $p\not\in U(\phi)$.
\begin{enumerate}
\item\label{lem:uram-1:1}  There exists a positive integer $A$ 
such that for every sequence of integers $k_1,\ldots,k_r$ that are multiples of $A$ we have 
that   $p \not \in U(\prod_{\PP^1} \phi^{{(q^{k_i})}_{\times}})$.
\item\label{lem:uram-1:2}  There exists a positive integer $B$ 
such that for every sequence of integers $k_1,\ldots,k_r$ that are multiples of  $B$, we have    $p \not \in U(\prod_{\PP^1} \phi^{{{k_i}}_{+}})$.
\end{enumerate}
\end{lemma}

\begin{proof}
By Lemma~\ref{lem:open-unr} with $S=\{p\}$, there exists a nonempty $p$-adic open set $V$ such that $p$ is unramified in $A^{\phi}_\zeta$, for all $\zeta\in V$. Fix some $\zeta\in V$.
By the $p$-adic continuity of the action of $\GL_2(\QQ_p)$ on $\mathbb{P}^1(\QQ_p)$, there exists an open neighborhood $ W \subset \GL_2(\QQ_p)$ of the identity matrix $I$ such that for any $g \in W \cap \GL_2(\QQ)$ we have that $g \zeta \in V $. 
In particular, $p$ is unramfied at $\phi^g(\zeta)$.
By Abhyankar's lemma, given any set of elements $g_1,\ldots g_n \in W$, 
$p$ is unramified at $\psi^{-1}(\zeta)$, for 
\[
\psi = \prod_{\mathbb{P}^1} \phi^{g_i}.
\]
  
Thus, for \ref{lem:uram-1:1},  it suffices to find a positive integer $A$ such that if $k$ is a multiple of $A$, then $g^{(q^k)_{\times}} \in W$. For this we  take $A= (p-1)p^m$ for a sufficiently large $m$.

Similarly, for \ref{lem:uram-1:2}, it suffices to find a positive integer $B$ such that if $k$ is a multiple of $B$, then $g^{k_{+}} \in W$. For this we take $B = p^\ell$ for a sufficiently large $\ell$.
\end{proof}

\begin{lemma}\label{lem:uram_S_add}
Let $\phi \colon C \to \PP^1$ be  a branched covering with rational branch locus.
Let $n \geq 1$ be an integer and $S$ a finite set of primes. Then, there exist a sequence of integers
$k_1,\dots,k_n$ such that 
\begin{align}
\label{lem:uram_S_add:1} &\RamG(\phi^{{{k_i}}_{+}}) \cap  \RamG(\phi^{{{{k_j}_{+}}}}) \subset \{\infty\}, & \mbox{for } i \neq j ,\\
\label{lem:uram_S_add:2}  &U(\prod_{\PP^1} \phi^{{{k_i}}_{+}})\cap S \subset  U(\phi).
\end{align}
\end{lemma}
\begin{proof}
Let $R \subset \mathbb{Q} = \PP^1(\QQ)\smallsetminus \{\infty\}$ be the finite branch points and 
\[
M = {\max R}-{\min R}
\]
the diameter of $R$.
As the set of finite  branch points of $\phi^{k_i^+}$ is $R+k_i$, to obtain \eqref{lem:uram_S_add:1}, it suffices to take the $k_i$'s such that $k_i-k_{i-1} > M$.

For every $p \in S \smallsetminus U(\phi)$, we let $B_p$ be the constant from Lemma~\ref{lem:uram-1} (2) (applied to $\phi$ and $p$). Then, to obtain \eqref{lem:uram_S_add:2},  it suffices to take the $k_i$'s to be multiples of $B_0=\prod_{p \in S \smallsetminus U(\phi)}B_p$.  Clearly, these two sufficient conditions can be simultaneously be satisfied; e.g., take  $k_i = i B$, where $B$ is a multiple of $B_0$ that is larger than $M$.
\end{proof}

\begin{lemma}\label{lem:uram_add}
Let $\phi \colon C \to \PP^1$ be  a branched covering with rational branch locus 
and let $n \geq 1$ be an integer. Then there exists a sequence of integers
$k_1,\dots,k_n$ such that both \eqref{lem:uram_S_add:1} and 
\begin{equation}
\label{lem:i22} U_{\infty}(\prod_{\PP^1} \phi^{{{k_i}}_{+}})  \subset  U_{\infty}(\phi)  
\end{equation}
hold true.
\end{lemma}
\begin{proof}
Denote 
\[
d = \#\RamG(\phi)
\]
and let $S$ be a finite set of primes that contains $V(\phi)$ and all the prime numbers $p\leq nd$.
Now choose $k_1,\dots,k_n$ as in  Lemma \ref{lem:uram_S_add}  and denote 
\[
\psi = \prod_{\PP^1} \phi^{{k_i}_{+}}.
\]
Thus \eqref{lem:uram_S_add:1} holds true. 

By \eqref{lem:uram_S_add:2} to obtain \eqref{lem:i22}, it suffices to show that 
\[
p \not \in S \Longrightarrow p \not \in U_{\infty}(\psi) .
\]
Indeed,  given $p\not\in S$, as $p>nd\geq \# \RamG(\psi)$, by \eqref{eq:Binvariant} we have $p \not \in B(\psi)$. 
Thus, by \eqref{eq:ubdbyvb} it remains to show that $p \not \in V(\psi)$: By \eqref{eq:abhy},
\[
V(\psi) = \bigcup V(\phi^{{{k_i}}_{+}})
\] 
and since  $p \not \in V(\phi)$ and 
\[
g_{{{k_i}}_{+}} \in \GL_2(\ZZ) \subseteq \GL_2(\ZZ_p) \cap \GL_2(\QQ),
\]
we also have $p \not \in V(\phi^{{{k_i}}_{+}})$ by \eqref{eq:Vinvariant}. Therefore, $p\not\in V(\psi)$ and by \eqref{eq:ubdbyvb} $p\not\in U_\infty(\psi)$. 
\end{proof}

\begin{lemma}\label{lem:uram_S_mult}
Let $\phi \colon C \to \PP^1$ be  a dominant map of curves with rational branch locus. Let $n \geq 1$ be an  integer, $q$ a rational prime, and $S$ be finite set of primes not containing $q$. Then, there exists a sequence of integers
$k_1,\ldots,k_n$ such that 
\begin{align}
\label{lem:i:1} &\RamG(\phi^{{q^{k_i}}_{\times}}) \cap  \RamG(\phi^{{q^{k_j}_{\times}}}) \subset \{0,\infty\},& \mbox{for } i \neq j,\\
\label{lem:i:2} &U(\prod_{\PP^1} \phi^{{q^{k_i}}_{\times}})\cap S \subset  U(\phi) .
\end{align}
\end{lemma}
\begin{proof}
Denote by $R  \subset \mathbb{Q}^{\times}=\PP^1(\QQ)\smallsetminus\{0,\infty\}$ the finite nonzero branch points and set 
\[
M = \max_{x\in R}\log_q|x|-\min_{x \in R}\log_q|x|.
\]
By \eqref{branchpoints}, \eqref{lem:i:1} would follow if  ${k_i-k_{i-1}} > M$.  

For every $p \in S \smallsetminus U(\phi)$, we let $A_p = A$ be the constant from Lemma~\ref{lem:uram-1}(1) (applied to $\phi$ and $p\neq q$). Then,   \eqref{lem:i:2} would follow if the $k_i$'s to be multiples of $A_0:=\prod_{p \in S \smallsetminus U(\phi)}A_p$. 
We thus put $k_i = i\cdot A$, where $A$ is a multiple of $A_0$ that is larger then $M$ to finish the proof.
\end{proof}

\begin{lemma}\label{lem:uram_mult}
Let $\phi \colon C \to \PP^1$ be  a dominant map of curves with branch locus defined over $\mathbb{Q}$, let $n \geq 1$ be an  integer, and let $q$ be a rational prime. Then, there exists a sequence of integers
$k_1,\dots,k_n$ such that both \eqref{lem:i:1} and 
\begin{equation}
\label{lem:i2:2}  U_{\infty}(\prod_{\PP^1} \phi^{{q^{k_i}}_{\times}})  \subset  U_{\infty}(\phi) \cup \{q\} 
\end{equation}
hold true. 
\end{lemma}
\begin{proof}
Denote 
\[
d = \#\RamG(\phi).
\]
Let $S$ be a finite set of primes $\neq q$ that contains $V(\phi)\cup \{p\leq nd\} \smallsetminus \{q\}$.
Take $k_1,\dots,k_n$ as in  Lemma \ref{lem:uram_S_mult}  and denote 
$$\psi = \prod_{\PP^1} \phi^{{q^{k_i}}_{\times}}.$$  
As \eqref{lem:i:1} holds true, it suffices to prove \eqref{lem:i2:2}. For this, by \eqref{lem:i:2}, it suffices to to show that if $p \not \in S$ and $p \neq q$, then
\[
p \not \in U_{\infty}(\psi) .
\]
Indeed, given $p\not\in S$ and $p\neq q$, we have $p >nd \geq \#\RamG(\psi)$, so by \eqref{eq:Binvariant}, $p \not \in B(\psi)$. 
By \eqref{eq:abhy}, 
\[
V(\psi)  = \bigcup V(\phi^{{q^{k_i}}_{\times}}).
\] 
As $p \not \in V(\phi)$ and 
\[
g_{{q^{k_i}}_{\times}} \in \GL_2(\ZZ_p) \cap \GL_2(\QQ),
\]
\eqref{eq:Vinvariant} gives that $p \not \in V(\phi^{{(q^{k_i})}_{\times}})$, so by \eqref{eq:ubdbyvb}, $p\not\in U_{\infty}(\psi)$, as needed. 
\end{proof}

\section{Irreducibility of Fiber Products and Group Theory}
We shall use the following function field criterion for irreducibility: Let  $\phi_1\colon C_1\to \mathbb{P}^1_{\QQ}$ and $\phi_2\colon C_2\to \mathbb{P}^1_{\QQ}$ be geometrically irreducible branched coverings with function field extensions $F_1/\QQ(T)$ and $F_2/\QQ(T)$, respectively, in some fixed algebraically closed field of $\QQ(T)$. The the fiber product $C_1\times_{\PP^1_\QQ} C_2$ is  irreducible (respectively geometrically irreducible) if and only if $F_1$, $F_2$ are linearly disjoint over $\QQ(T)$ (respectively $F_1\bar\QQ$ and $F_2\bar\QQ$ are linearly disjoint over $\bar\QQ(T)$).

\begin{lemma}\label{lem:irr_one_br}
Let $\phi_i\colon C_i\to \mathbb{P}^1_{\QQ}$ be a geometrically irreducible branched covering, $i=1,2$. Assume that $\RamG(\phi_1)\cap \RamG(\phi_2)\subseteq\{\alpha\}$ for some $\alpha\in \PP^1(\QQ)$. Then $C_1\times_{\PP^1} C_2$ is geometrically irreducible. 
\end{lemma}

\begin{proof}
Let $F_1/\QQ(T)$ and $F_2/\QQ(T)$ be the function fields extensions corresponding to $\phi_1,\phi_2$ in some algebraic closure of $\QQ(T)$. Let $E_i = F_i \bar{\QQ}$ be the base change to an algebraic closure $\bar{\QQ}$ of $\QQ$ and let $N_i$ be the Galois closure of $E_i$ over $\bar{\QQ}(T)$, $i=1,2$. 

By Abhyankar's lemma, $N_i$ has the same branch locus as $F_i$, and so $N_1\cap N_2$ is ramified at $\RamG(\phi_1)\cap \RamG(\phi_2)$ which consists, by assumption, of at  most one point. By the Riemann-Hurwitz formula, $N_1\cap N_2=\bar{\QQ}(T)$.

Thus $N_1,N_2$ are linearly disjoint over $\bar{\QQ}(T)$, which implies that the subextensions $E_1,E_2$ are also linearly disjoint. Thus, $C_1 \times_{\PP^1_{\QQ}} C_2$ is geometrically  irreducible.
\end{proof}

In the applications below, we need to relax the condition of Lemma~\ref{lem:irr_one_br} that the branch loci of $\phi_1$ and of $\phi_2$ have at most one rational point in common. For this we need some group theory.

\begin{definition}\label{def:Sp}
For a prime number $p$, we say that a finite group $G$ satisfies \emph{condition-$E(p)$} if all the nontrivial simple quotients of $G$ are of order $p$, but none of the quotients of the commutator $[G,G]$ are. 
\end{definition}

We give a few examples and basic properties and we omit the details: 
\begin{enumerate}
\item Let $G$ be an $E(p)$-group. Then $G$ is a $p$-group if and only if $G$ is abelian.
\item The symmetric group $S_m$ is $E(2)$. 
\item Let $m$ be a positive integer with $v_2(m)\leq 1$. Then, the Dihedral group $D_m$ of order $2m$ is $E(2)$. 
\item If $G,H$ satisfy condition-$E(p)$, then so does $G\times H$. 
\item Let $G$ be a group satisfying condition-$E(p)$ and $N$ a normal subgroup. Then $G/N$ satisfies condition-$E(p)$. (Indeed, $[G/N,G/N]=[G,G]N/N$.)
\item Let $G$ be an $E(p)$-group and $H$ a prefect group, then the wreath product $H\wr G$ satisfies $E(p)$. The proof of this fact is slightly involved, but we omit it, as we do not use.
\end{enumerate}

We study irreducibility of fiber products of covers with $E(p)$-Galois groups. For this we need an auxiliary result from group theory. 

\begin{lemma}\label{lem:grptheory}
Let $p$ be a prime,  $G_1, \ldots, G_n$ groups that satisfy condition-$E(p)$, put $\Phi(G_i):=G_i^p[G_i,G_i]$
and 
\[
\psi \colon G_1\times \cdots \times G_n \to (G_1/\Phi(G_1)) \times \cdots \times (G_n/\Phi(G_n))
\]
the quotient map. Let $H\leq G_1\times \cdots \times G_n$ be such that the restriction of the projection on the $i$-th coordinate to $H$, $\pi_i\colon H \to G_i$ is surjective, for every $i=1,\ldots, n$ and  the restriction of $\psi$ to $H$ is surjective.
 Then $H=G_1\times \cdots \times G_n$.
\end{lemma}

\begin{proof}
Since the family of finite groups satisfying condition-$E(p)$ is close under  direct products and since $\Phi$ respects direct products, by induction argument, we may assume that $n=2$. 
Let $K_i = \ker \pi_i$ and $C_i = \pi_i^{-1} ([G_i,G_i])$, $i=1,2$. Note that $K_i\leq C_i$ are normal in $H$ and that $H/K_i\cong G_i$ and $C_i/K_i = [H/K_i:H/K_i]$. Let 
\[
\rho\colon G_1\times G_2\to G_1^{ab} \times G_2^{ab}
\]
be the abelianization map. We break the proof into several parts. 

\item\textsc{Part 1:} $\rho|_{H}$ is surjective. Indeed, by assumption, $\Phi(G_i)/[G_i,G_i]$ is the Frattini subgroup of $G_i^{ab}$. Thus the assumption gives that $\rho(H)$ generates $G_1^{ab} \times G_2^{ab}$ modulo the Frattini subgroup; hence $\rho(H)=G_1^{ab}\times G_2^{ab}$. 

\item\textsc{Part 2:} $C_1C_2 = H$. 
Indeed, it is immediate that $C_1 = \rho|_H^{-1} (1\times G_2^{ab})$ and $C_2 = \rho|_H^{-1} (G_1^{ab}\times 1)$. Hence, as $\rho|_{H}$ is surjective, $C_1C_2=\rho|_{H}^{-1} (G_1^{ab}\times G_2^{ab}) = H$.

\item \textsc{Part 3:} $H=K_1C_2$. Indeed, as $H = C_1 C_2 = C_1 (K_1 C_2)$, the second isomorphism theorem gives that
\[
H/K_1 C_2 \cong C_1/ C_1 \cap (K_1C_2). 
\]
Assume by contradiction that $H/K_1C_2$ is nontrivial; then $H/K_1C_2$ has a simple quotient $S$.
As $H/K_1C_2$ is a quotient of $H/C_2\cong G_2^{ab}$, $S$ is of order $p$. On the other hand, 
$C_1/ C_1 \cap (K_1C_2)$ is a quotient $C_1/K_1 \cong [G_1,G_1]$, which contradicts the assumption that $G$ satisfies condition-$E(p)$.

\item \textsc{Part 4:} $H=K_1K_2$. We argue in a similar fashion as in Part 3: As $H= K_1 C_2 = (K_1 K_2) C_2$, the second isomorphism theorem gives that
\[
H/K_1K_2 = C_2/ (C_2 \cap K_1K_2).
\]
Assume by contradiction that $H/K_1K_2$ is nontrivial, then it has a simple quotient  $S$.
Since $H/K_1K_2$ is a quotient of $H/K_2\cong G_2$ and $G_2$ satisfies condition-$E(p)$, the order of $S$ is $p$. On the other hand, $C_2/ (C_2 \cap K_1K_2)$ is quotient of $C_2/K_2\cong [G_2,G_2]$, which contradicts the assumption that $G_2$ satisfies condition-$E(p)$.

\item \textsc{Conclusion of the proof:}
Since $H=K_1K_2$ and $K_1\cap K_2 = 1$, we get that 
\[
H\cong K_2\times K_1 \cong H/K_1\times H/K_2 \cong G_1\times G_2,
\]
as needed.
\end{proof}

\begin{lemma}\label{lem:irr_M}
Let $p$ be a prime and for each $i=1,\ldots, n$ let $\phi_i\colon C_i\to \mathbb{P}^1_{\QQ}$ be a geometrically irreducible branched covering that is generically Galois with Galois group $G_i$. 
Let $D_i=C_i/\Phi(G_i)$, where $\Phi(G_i) = G_i^{p}[G_i,G_i]$. Assume that $G_i$ satisfies condition-$E(p)$ for all $i$ and that $\prod_{\PP^1_{\QQ}} D_i$ is geometrically irreducible. Then $\prod_{\PP^1_{\QQ}} C_i$ is geometrically irreducible.
\end{lemma}

\begin{proof}
For each $i$, let $\bar{\QQ}(T) \subseteq E_i \subseteq F_i$ the function field extensions corresponding to the maps $\PP^1_{\bar{\QQ}}\leftarrow (D_i)_{\bar{\QQ}}  \leftarrow (C_i)_{\bar{\QQ}}$.
Since $(C_i)_{\bar{\QQ}}$ is irreducible by assumption, it follows that $(D_i)_{\bar{\QQ}}$ is also irreducible. Hence by Galois correspondence and since $\Phi(G_i)\lhd G_i$ it follows that these extensions are Galois with Galois groups 
\begin{align*}
\Gal(F_i/\bar{\QQ}(T))= G_i, && \Gal(F_i/E_i)=\Phi(G_i), && \Gal(E_i/\bar{\QQ}(T))\cong G_i/\Phi(G_i).
\end{align*}
Let $E = E_1\cdots E_n$ be the composition of $E_i$, $i=1,\ldots, n$. 
The assumption that $\prod_{\QQ(T)} D_i$ is absolutely irreducible, 
implies that 
\[
[E:\bar{\QQ}(T)]=\prod [E_i:\bar\QQ(T)]=\prod[G_i:\Phi(G_i)].
\]
Hence, $\Gal(E/\bar{\QQ}(T))\cong \prod_{i} G_i/\Phi(G_i)$. 
We put $F=F_1\cdots F_n$. We summarize the above  in Diagram~\ref{diag}.

\begin{figure}[H]\renewcommand{\figurename}{Diagram}
\[
\xymatrix@C40pt{
&E\ar@{.}[dl]_{\prod_i{G/\Phi(G_i)}} \ar@{-}[r]& F\\
\bar{\QQ}(T)\ar@{-}[r]_{G_i/\Phi(G_i)}&E_i\ar@{-}[r]_{\Phi(G_i)}\ar@{-}[u]&F_i\ar@{-}[u]
}
\]
\caption{Function Fields and Galois Groups}\label{diag}
\end{figure}

Let $H =\Gal(F_1F_2/\bar{\QQ}(T))$. Then $H$ embeds into $\prod_{i} G_i$ via the restriction maps; namely, $\sigma\mapsto (\sigma|_{F_i})_i$. 
The restriction of the projection onto the $j$th coordinate $\prod_{i} G_i \to G_j$ to $H$  is surjective for every $j$. Also, by Galois correspondence, the image of $H$ under the quotient map $\prod_i G_i \to \prod_{i} G_i/\Phi(G_i)$ is $\Gal(E/\bar{\QQ}(T)) = \prod_{i} G_i/\Phi(G_i)$. Thus the conditions of Lemma~\ref{lem:grptheory} are satisfied, so $H= \prod_{i} G_i$.  This implies that, $[F:\bar\QQ(T)] = \deg \phi_1\times_{\PP^1_{\QQ}}  \cdots  \times_{\PP^1_{\QQ}}  \phi_n$, so $C_1\times_{\PP^1_{\QQ}}\cdots \times_{\PP^1_{\QQ}}  C_n$ is geometrically irreducible.
\end{proof}

\subsection{The $E(p)$-Condition and Rational Rigid Tuples}
\label{sec_Rigid}
Let $G$ be a finite group. We say that a  $k$-tuple $\mathbf{g} =(g_1,\dots,g_k) \in G^{k}$ is a \emph{good generating $k$-tuple} for $G$ if $G$ is generated by $g_1,\ldots, g_k$ and $g_1\cdots g_k = 1$.
Two good generating $k$-tuples  $\mathbf{g} =(g_1,\dots,g_k)$ and $\mathbf{g}' =(g_1',\dots,g_k')$ for $G$ are \emph{semi-conjugate} if for every $1 \leq i \leq k$ there exists $h_i \in G$ such that $g_i' = h_i^{-1}g_ih_i$. We say that $\mathbf{g}$ and $\mathbf{g}'$ are \emph{conjugate} if there exists $h \in G$ such that $g_i' = h^{-1}g_ih$ for all $1 \leq i \leq k$.

Let $G$ be a finite group,  a $k$-tuple $\mathbf{g} =(g_1,\dots,g_k) \in G^{k}$ is called  \emph{rigid} if the following conditions hold:
\begin{enumerate}
\item\label{i:rigid:center} $G$ has a trivial center.
\item\label{i:rigid:gen}$\mathbf{g}$ is a good generating tuple.
\item\label{i:rigid:rig} Every  good generating $k$-tuple $\mathbf{g}'$ which is semi-conjugate to $\mathbf{g}$ is conjugate to $\mathbf{g}$.
\end{enumerate}
Recall that an element $g$ in a group $G$ is called \emph{rational} if for every integer $n$ which is relatively prime to the order of $G$, $g^{n}$ is conjugated to $g$.
A rigid  tuple is called \emph{rational rigid} if in addition:
\begin{enumerate}
\item[4.] \label{i:rigid:rational} Every $g_i$ is rational. 
\end{enumerate}

\begin{lemma}\label{l:rig:addone}
If $\mathbf{g} =(g_1,...,g_k)$ is a rational rigid $k$-tuple for $G$, then $\mathbf{g'} =(g_1,\dots,g_i,1,g_{i+1},\dots,g_k)$ is a rational rigid $k+1$-tuple.
\end{lemma}

\begin{proof}
Clear. 
\end{proof}

\begin{lemma}\label{l:rig:prod}
Let $G$ and $H$ be finite groups. Let $\mathbf{g} =(g_1,\ldots,g_k)$ be a rational rigid $k$-tuple for $G$ and 
$\mathbf{h} =(h_1,...,h_k)$ be a rational rigid $k$-tuple for $H$.  Assume that the collection of elements $(g_i,h_i) \in G \times H$ generates $G\times H$. Then $\mathbf{g} \times \mathbf{h} = ((g_1,h_1),\dots,(g_k,h_k)) \in (G\times H)^k$ is a rational rigid $k$-tuple for $G \times H$.
\end{lemma}
\begin{proof}
The rationality is clear. 
Condition~\ref{i:rigid:center} is clear since the center of a product is the product of centers. 
Condition~\ref{i:rigid:gen} holds true by assumption. 

Hence it suffices to show Condition~\ref{i:rigid:rig}: Indeed. let $\mathbf{g}' \times \mathbf{h}'  \in (G \times H)^k$  be a  good generating tuple which is semi-conjugate to $\mathbf{g}\times \mathbf{h}$. Then 
$\mathbf{g}'$ is a  good generating tuple which is semi-conjugate to $\mathbf{g}$ and $\mathbf{h}'$ is a  good generating tuple which is semi-conjugate to $\mathbf{h}$. Thus,  $\mathbf{g}'$ is conjugate to $\mathbf{g}$ and $\mathbf{h}'$ is conjugate to $\mathbf{h}$. This implies that $\mathbf{g}' \times \mathbf{h}'$ is conjugate to $\mathbf{g} \times \mathbf{h}$.
\end{proof}

\begin{proposition}\label{prop:rig:S_p}
Let $G_1$,$G_2$ be groups satisfying the $E(p)$-condition. Assume that $G_i$ admits a rational rigid $k_i$-tuple for each $i=1,2$. Let $d_i=d(G_i^{ab})$. Then $G_1
\times G_2 $   admits a   rational rigid $s$-tuple, for 
s = $d_1+d_2 + \max(k_1-d_1,k_2-d_2)$
\end{proposition}
\begin{proof}
Since $G_i$ is $E(p)$ we have that $G_i^{ab}$ is a $p$-group and $G_i/\Phi(G_i) = G_i/G_i^p[G_i,G_i] = (\mathbb{Z}/p\mathbb{Z})^{d_i}$.
Let 
\[
\rho_i : G_i \to  G_i/\Phi(G_i) = (\mathbb{Z}/p\mathbb{Z})^{d_i}
\]
be the quotient map.
By Lemma~\ref{l:rig:addone}, we may assume w.l.o.g.\ that $r:=k_1-d_1=k_2-d_2$, so $s= d_1+d_2+r$.
let $\mathbf{g}^{(i)} =(g^{(i)}_1,...,g^{(i)}_{k_i})$ be a rational rigid $k_i$-tuple for $G_i$.
Let $A_i \subset \{1,\dots,k_i\}$ be a set of size $d_i=|A_i|$ such that 
$\{\rho_i(g^{(i)}_a):a\in A_i\}$ generates $G_i/\Phi(G_i)$ and $B_i=\{1,\ldots, k_i\}\smallsetminus A_i$ the complement. Write the elements of $B_i$ as 
\[
b_{i,1}<b_{i,2} < \ldots < b_{i,r}.
\]
Consider all the pairs 
\[
v_a=(g^{(1)}_{a},1), \quad v'_{a'} = (1,g^{(2)}_{a'}), \quad  w_{j} = (g^{(1)}_{b_{1,j}},g^{(2)}_{b_{2,j}}),
\]
for $a\in A_1$, $a'\in A_2$, and $j=1, \ldots, r$. One may order them such that the resulting $s$-tuple $V$ of elements in $(G_1\times G_2)^s$ has the property that the projection to each of the coordinates $G_i$ gives the original tuple diluted by $1$'s. 

Let $H\leq G_1\times G_2$ be the subgroup generated by $V$. 
By Lemma~\ref{l:rig:prod}, it suffices to show that $H=G_1\times G_2$. Indeed, on the one hand, $H$ maps onto each of the $G_i$'s. On the other hand, by the construction of $V$, $(\rho_1\times\rho_2) (V)$ contains a basis of $G_1/\Phi(G_1) \times G_2/\Phi(G_2)$, so by Lemma~\ref{lem:grptheory}, $H=G_1\times G_2$, as needed for rigidity. The rationality is immediate.  
\end{proof}

Applying the previous proposition repeatedly gives:
\begin{corollary}\label{eq:corrigid}
Let $G$ be a group satisfying the $E(p)$-condition, $d=d(G^{ab})$, and $n\geq 1$. Assume that $G$ admits a rational rigid $r$-tuple. Then $G^n$ admits a rational rigid $s$-tuple, for 
$s = (n-1)d + r$.
\end{corollary}

\section{The Minimal Ramification Problem}
In this section we prove the asymptotic inequalities~\eqref{eq:generalG}-\eqref{rigidrigidrigid} basing on the methods developed so far.

\begin{definition}\label{def:MMM}
Let $G$ be a finite group, $U$ a finite set of primes of $\QQ$ and $\bfd=(d_1,\ldots,d_r)$ a tuple of positive integers.
We say that $G$ has $(U;\bfd)$ realization if there exists a geometrically irreducible  branched covering $\phi\colon C\to \PP^1_\QQ$ such that
\begin{itemize}
\item $\QQ(C)/\QQ(\PP^1)$ is Galois with Galois group $G$,
\item $U(\phi)\subseteq U$,
\item $\RamG(\phi) = \{(D_1), \ldots, (D_r)\}$ with $D_i(t,s)\in \ZZ[t,s]$ homogenous of degree $d_i$.
\end{itemize}
\end{definition}

\begin{proposition}\label{prop:formulation for MRP}
Let $G$ be a finite group that has a $(U;\bfd)$ realization and let $L/\QQ$ be a finite extension. Then there exists a Galois extension $N/\QQ$ with Galois group $G$ such that $N\cap L=\QQ$ and $\#\Ram_U(N/\QQ) \leq B(\mathbf{d})$, where $B(\mathbf{d})$ is defined in \eqref{eq:Bd}.
In particular,
\[
m(G)\leq  B(\bfd)+\#U .
\]
\end{proposition}

\begin{proof}
Let $\phi\colon C\to \PP^1(\QQ)$ be a branched covering from Definition~\ref{def:MMM} and let $Z$ be the set of $[a:b]\in \PP^1(\QQ)$ such that $A^{\phi}_{[a:b]}\otimes L$ is not a field. Then $Z$ is thin (see~\cite[Corollary~12.2.3]{FriedJarden2009} and note that a subset of $\QQ$ is thin if and only if its complement contains a Hilbert set by Lemma~13.1.2 in \emph{loc.cit.}). By Theorem~\ref{thm:main} and by \eqref{eq:Bd}, there exists $[a:b]\in\PP^1(\QQ)\smallsetminus Z$ such that for $N=A^{\phi}_{[a:b]}$ we have
\[
\#\Ram_U(N/\QQ) \leq B(D_1,\ldots, D_r)\leq B(\mathbf{d}).
\] 
As $N\otimes L$ is a field, it follows that $N$ is a field that is linearly disjoint from $L$, and so $N\cap L=\QQ$. Clearly $N/\QQ$ is Galois with Galois group $G$.
\end{proof}

\begin{proof}[Proof of~\eqref{eq:generalG}] By assumption $G$ has a $(U;\bfd)$ realization for some $U,\bfd$. We claim that we can realize $G^n$ with at most $B(\bfd) n$ ramified primes outside of $U$. And indeed, assume by induction that $G^{n-1} = \Gal(L/\QQ)$ and  $\#\Ram_U(L/\QQ)\leq B(\bfd) (n-1)$. Then by Proposition~\ref{prop:formulation for MRP} we have a Galois extension $N/\QQ$ with Galois group $G$ such that $N\cap L=\QQ$ and $\#\Ram_U(N/\QQ) \leq B(\mathbf{d})$, so $NL/\QQ$ is a Galois extension with Galois group $G^{n}=G \times G^{n-1}$ and
\[
\begin{split}
\#\Ram_U(NL/\QQ) &= \#\Ram_U(N/\QQ) + \#\Ram_U(L/\QQ) \\ & \leq B(\bfd)  + B(\bfd) (n-1)=B(\bfd)n.
\end{split}
\]
In particular we have
\begin{equation}\label{eq:numberofram}
m(G^n) \leq B(\bfd) n + \#U = O(n),
\end{equation}
as needed.
\end{proof}

We remark that if $G$ has a $(U;\bfd)$ realization with $\bfd = \mathbf{1}_r = \overbrace{(1, \ldots, 1)}^{\mbox{\tiny $r$ times}}$, then since $B(\bfd)\leq r$ by~\eqref{eq:B-GTZ-bd}, the inequality  \eqref{eq:numberofram} immediately gives that
\[
m(G^n) \leq rn + O(1).
\]
However this is not sufficient for \eqref{eq:generalGrat}, as we need to reduce $r$ to $r-1$. So to prove \eqref{eq:generalGrat} one requires an extra construction:

\begin{proposition}\label{prop:Gn}
Let $\mathbf{1}_r = (1,\ldots, 1)$ be an $r$-tuples of ones, let $G\neq 1$ be a finite group having a $(U;\bfd)$ realization, and let $n\geq 1$ be an integer. Then $G^n$ has a
$(U, \mathbf{1}_R)$ realization, where $R=(r-1)n+1$.
\end{proposition}

\begin{proof}
Let $\phi\colon C\to \PP^1$ be the $(U;\bfd)$ realization of $G$ with $\RamG(\phi)=\{(D_1), \ldots, (D_r)\}$, $\deg D_i=1$. Since $G\neq 1$, the morphism $\phi$ must be ramified, so $r\geq 1$. Without loss of generality we may assume that $\infty$ is a branch point (otherwise we compose $\phi$ with a matrix in $\GL_2(\QQ)$ that maps a branch point to infinity).

Put $S = V(\phi)\cup \{p\leq R\}\cup\{\infty\} $, where $V(\phi)$ is the set of vertically ramified primes of a model of $\phi$ over $\mathbb{Z}$. We apply Lemma~\ref{lem:uram_S_add} to get integers $k_1, \ldots, k_n$ satisfying \eqref{lem:uram_S_add:1} and \eqref{lem:uram_S_add:2}.
Put $\hat\phi=\prod_{\PP^1} \phi^{k_{i+}} \colon \hat{C} \to \PP^1_{\QQ}$. 

By \eqref{branchpoints}, $\RamG(\phi^{k_i+}) = \{\infty, p_{1,i}, \ldots, p_{r-1,i}\}$. By \eqref{lem:uram_S_add:1}, $p_{j,i}\neq p_{j',i'}$ for all $(j,i)\neq (j',i')$. By Abhyankar's lemma, we conclude that $\RamG(\hat\phi) = \{\infty, p_{1,1}, \ldots, p_{r-1,n}\}$. In particular, $\hat\phi$ has exactly $R$ branch points which are all $\QQ$-rational. 

The conditions of Lemma~\ref{lem:irr_one_br} are satisfied by \eqref{lem:uram_S_add:1}, thus the curve $\hat C$ is geometrically irreducible. This in particular implies that the extensions $E_i/\QQ(\PP^1)$ defined by $\phi^{k_{i+}}$  are linearly disjoint Galois extensions of $\QQ(\PP^1)$, and so the Galois group of $\QQ(\hat C)=\prod E_i$ over $\QQ(\PP^1)$ is the direct product of the Galois groups of the extensions; i.e., $G^n$.   

By \eqref{eq:Vinvariant}, $V(\phi^{k_{i+}}) = V(\phi)$; so by \eqref{eq:abhy} we have 
\begin{equation}\label{eq:Vhatphi}
V(\hat\phi)=V(\phi)\subseteq S.
\end{equation}
By \eqref{eq:Binvariant}, $B(\hat\phi) \subseteq \{p\leq R\} \subseteq S$. 
Together with \eqref{eq:Vhatphi} and  \eqref{eq:ubdbyvb} this gives that 
\[
U(\hat\phi) \subseteq  S.
\]
Since $U(\phi)\subseteq U$ and by \eqref{lem:uram_S_add:2} we conclude that 
\[
U(\hat\phi) = U(\hat\phi)\cap S \subseteq U(\phi)\cap S \subseteq U,
\]
and so $\hat\phi$ is a $(U, \mathbf{1}_{R})$ realization of $G^n$, as needed.  
\end{proof}

\begin{proof}[Proof of~\eqref{eq:generalGrat}]
Assume $G$ has a $(U;\mathbf{1}_r)$ realization. Then by Proposition~\ref{prop:Gn}, 
$G^n$ has a $(U;\mathbf{1}_R)$, $R=(r-1)n+1$ realization. By Proposition~\ref{prop:formulation for MRP} and the bound \eqref{eq:B-GTZ-bd} we get  
\[
m(G^n) \leq \#U + B(\mathbf{1}_R) \leq R + \#U \leq (r-1)n + \#U + 1 =(r-1)n +O(1).
\]
This finishes the proof. 
\end{proof}

Next we prove \eqref{prime_quotient}, which reduces the number of ramification to $(r-2)n$ under certain group theoretical conditions.

\begin{proof}[Proof of \eqref{prime_quotient}]
Let $\phi\colon C\to \PP^1_\QQ$ be a non-constant map of smooth connected projective $\QQ$-curves that is generically Galois with group $G$. Assume that the branch locus consists of $r$ rational points. 
We assume that $[G,G]$ is simple non-abelian, $d=d(G^{ab})\leq r-2$, and that there exists a prime number $p$ such that every maximal normal subgroup has index $p$. 
This implies that $d(G^{ab})$ is a $p$-group, and that $G$ satisfies condition $E(p)$.
We note that in this case the Frattini quotient of $G^{ab}$ is $G/M(G)\cong (\mathbb{Z}/p\mathbb{Z})^d$ with $M(G) = G^p[G,G]$; and thus a subgroup $H\leq G^n$ maps onto $(G^{ab})^n$  if and only if it maps onto $(G/M(G))^n$.

We let $C_M=C/M(G)$; so $\phi_M\colon C_M\to G$ is Galois with Galois group $(\mathbb{Z}/p\mathbb{Z})^d$. Choose $d$ branch points $x_1,\ldots, x_d$, such that the inertia groups above the $x_i$'s generate $(\mathbb{Z}/p\mathbb{Z})^d$. By assumption, there exist at least two other branch points $y_1,y_2$. By applying a Mobius transformation, we may assume w.l.o.g.\ that $y_1=[0:1]$ and $y_2=[1:0]$. 

We pick an auxiliary prime $q$. By Lemma~\ref{lem:uram_mult} there exist $k_1,\ldots, k_n$ such that if we write $C_i=C$ and $\phi_i=\phi^{q^{k_i}_\times}\colon C_i\to \PP^1_{\QQ}$, 
then we have (for $i\neq j$)
\[
\RamG(\phi_i)\cap \RamG(\phi_j) \subseteq \{0,\infty\} \quad \mbox{and}\quad U_{\infty}(\prod_{\PP^1} \phi_i) \subseteq U_{\infty}(\phi)\cup \{q\}.
\]
Let $F_i/\QQ(x)$ be the function field extension corresponding to $\phi_i\colon C\to \PP^1_{\QQ}$, $i=1,\ldots, n$. Then $G\cong\Gal(F_i/\QQ(x))$; denote by $L_i$ and $L_i'$ the fixed fields of $[G,G]$ and $G^p[G,G]$ (respectively) in $F_i$. Since each $\Gal(L_i'/\QQ(x))$ is generated by the inertia groups over distinct points, the $L_i'$ are linearly disjoint over $\QQ(x)$. 
By Lemma~\ref{lem:irr_M}, 
$\hat{C}=\prod_{\PP^1} C_i$ is geometrically irreducible. Now, as we chose the $k_i$ as in Lemma~\ref{lem:uram_mult}, we have that $U_{\infty}(\prod_{\PP^1}\phi_i) \subseteq U_{\infty}(\phi)\cup\{q\}$. By construction, the  branch locus of $\hat\phi$ consists of $(r-2)n+2$ rational branch points. So, if we put $U=U_{\infty}(\phi)\cup\{q\}\cup \{\infty\}$ and $s=(r-2)n+2$, we have obtained a 
\[
(U;\mathbf{1}_{s})
\]
realization of $G^n$. By \eqref{eq:B-GTZ-bd} and Proposition~\ref{prop:formulation for MRP}, 
\[
m(G^n)\leq s +\#U \leq (r-2)n + \#U_{\infty}(\phi) + 4 = (r-2) n +O(1),
\] 
which proves \eqref{prime_quotient}.
\end{proof}

Now we consider the special case $G=S_m$ and we prove \eqref{SN-realization}, that is $m(S_m^n)\leq n+4$ and \eqref{SN-realizationn=1}, $m(S_m)\leq 4$. For this we first need to recall a concrete realization of $S_m$ over $\PP^1_{\QQ}$.

\begin{lemma}\label{lemma:constructing_SN}
Let $a,b,c\in \PP^1(\QQ)$ be distinct and $m>3$. 
There exists a cover $\phi\colon C\to \PP^1_{\QQ}$ with Galois group $S_m$ such that $\RamG(\phi) = \{a,b,c\}$ and the inertia groups at $a,b,c$ are generated by cycles of length $n,n-1,2$ respectively and $U(\phi)=\{\infty\}$. 
\end{lemma}

\begin{proof}
By applying Mobius transformation, we see that it suffices to find $\phi$ for one triplet $(a,b,c)$. 
Consider $\PP^1\to \PP^1$ given by $x\mapsto x^m-x^{m-1}$, i.e.\ generated by $f(X,Y) = X^m-X^{m-1} -Y$, let $F$ be the splitting field of $f$ over $\QQ(Y)$, and let $\phi \colon C\to \PP^1_{\QQ}$ be the branch covering corresponding to $F/\QQ(Y)$. It is an exercise to show that the Galois group is $S_m$ and that ramification points are $0,u,\infty$, with $u = \frac{m-1}{m}$ and that the inertia groups are generated by cycles of lengths $2,n-1,n$, respectively. For details see \cite[Page~42]{SerreTopics}.

It now remains to calculate $U=U(\phi)$. For any $y\in \QQ\smallsetminus \{0,u\}$, let $A_y$ be the algebra at $y$. Since 
the $X$-derivative of $f(X,y)$ has only $2$ roots ($0$ and $u$), $f(X,y)$ has at most $3$ real roots. Thus $A_y$ has at most three embeddings into $\RR$, which implies as $m>3$ that $A_y\otimes \RR\not\cong \RR^3$. Thus $\infty\in U$. 

A direct application of the discriminant formula $\disc f = \pm m^m \prod f(\alpha)$, where $\alpha$ runs on the set of zeros of $f'$ with multiplication,  shows that 
\[
\disc f(X,y) = \pm Y^m \big( (m-1)^{m-1} + m^m Y\big).
\]
Let $p$ be a prime; we show that there exists $y\in \ZZ$ with $p\nmid \disc f(X,y)$, and thus $p\not \in U$. This will show that $U=\{\infty\}$. 
If $p\nmid m$ and $p> 2$, then $m^m y^m$ takes $p-1>1$ values for $y\not \equiv 0\pmod p$, and so we can take $y\in\ZZ$ with  $m^m y \not\equiv -(m-1)^{m-1},0\mod p$; so $p\nmid \disc f(X,y)$, as needed. If $p\mid m$, then $p\nmid \disc f(X,1)$. We are left with the case $p=2$ and $m$ odd; then $p\mid m-1$, so $p\nmid \disc f(X,1)$.   
\end{proof}

\begin{proof}[Proof of \eqref{SN-realization} and \eqref{SN-realizationn=1}]
We just apply the construction appeared in the proof of \eqref{prime_quotient} to the cover  $\phi\colon C\to \PP^1$ given in Lemma~\ref{lemma:constructing_SN} that is ramified at $(\infty,0,1)$ with the inertia group at $1$ being generated by a transposition.

This gives a $(U,\bfd_{n+2})$ realization of $S_m^n$, with $U=\{\infty\}$ if $n=1$ and $U\subseteq \{\infty,q\}$ if $n\geq 2$. Thus by \eqref{eq:B-GTZ-bd},   Proposition~\ref{prop:formulation for MRP} gives that 
\[
m(S_m^n) \leq n+4,
\]
as needed for \eqref{SN-realization}, and that 
\[
m(S_m)\leq 4,
\]
as needed for \eqref{SN-realizationn=1}.
\end{proof}

We conclude by proving our results for rational rigid groups.

\begin{proof}[Proof of \eqref{rigid}]
Let $G$ be a group with a rational rigid $r$-tuple. 
By \cite[Theorem~8.1.1]{SerreTopics}, there exists a geometrically irreducible branched covering $\phi\colon C\to \mathbb{P}^1_\QQ$ with  $\RamG(\phi)=\{1,\ldots, r\}$.  Let $T= \{p\leq r\}\cup \Prms(|G|)$. If $p\not\in T$,  by  
\cite[Theorem~1.2]{Beckmann91}, $p$ is unramified at $A^{\phi}_{r+1}$, $r+1\in \AA^1(\QQ)\subseteq \PP^1(\QQ)$. So $U(\phi)\subseteq T$. Now Proposition~\ref{prop:formulation for MRP} and \eqref{eq:B-GTZ-bd} immediately gives $m(G) \leq r+\#T$.
\end{proof}

\begin{proof}[Proof of \eqref{rigidrigidrigid}]
By Corollary~\ref{eq:corrigid}, $G^n$ has a rational rigid $s$-tuple with $s=d(G^{ab})(n-1) + r = d(G^{ab}) n + O(1)$. Note that by the prime number theorem $\#\{ p \leq s\} = O(n/\log n)$ and that $\Prms(|G|^n)=\Prms(|G|)=O(1)$. Hence \eqref{rigid} gives that
\[
m(G) \leq d(G^{ab}) n + O\left(\frac{n}{\log (n)}\right).
\]
\end{proof}

\subsection*{Acknowledgments}
We thank Olivier Ramar\'e for introducing to us the theory of weighted sieve and to  Yonatan Harpaz and Arno Fehm for helpful discussions. 

The first author was partially supported by the Israel Science Foundation, grant no.\ 952/14.

\bibliographystyle{plain}

\end{document}